\documentclass[12pt]{amsart}

\usepackage{cite}

\usepackage{amsmath}
\usepackage{amstext}
\usepackage{amssymb}

\usepackage{amsthm}

\theoremstyle{plain}
\newtheorem{theorem}{Theorem}[section]
\newtheorem{lemma}[theorem]{Lemma}
\newtheorem{corollary}[theorem]{Corollary}

\theoremstyle{definition}
\newtheorem{definition}{Definition}[section]

\numberwithin{equation}{section}

\newcommand{\conj}[1]{\overline{#1}}
\DeclareMathOperator{\sgn}{sgn}

\DeclareMathOperator{\Rp}{Re}

\begin{document}

\title{Bergman-H\"{o}lder Functions, Area Integral Means and Extremal Problems}
\author{Timothy Ferguson}
\address{Department of Mathematics\\University of Alabama\\Tuscaloosa, AL}
\email{tjferguson1@ua.edu}

\date{\today}

\begin{abstract}
We study certain weighted area integral means of analytic functions in the 
unit disc. We relate the growth of these means to the property of being
mean H\"{o}lder continuous 
with respect to the 
Bergman space norm.  In contrast with earlier work, we use the 
second iterated difference quotient instead of the first.  
We then 
give applications to Bergman space extremal problems.
\end{abstract}

\maketitle

This paper deals with mean H\"{o}lder type 
smoothness conditions for functions in 
Bergman spaces on the unit disc $\mathbb{D}$, 
and the relation of these conditions to extremal problems 
in Bergman spaces. 

The first main topic is area integral means and smoothness conditions. 
It is well known (due to Hardy and Littlewood) 
that $f$ is analytic in the unit disc and 
$|f'(re^{i\theta})| \leq C(1-r)^{-1+\beta}$ for $0 < \beta \leq 1$ if and only if 
$f$ is continuous in the closed unit disc and 
$|f(e^{i\theta + it}) - f(e^{i\theta})| \leq C'|t|^\beta$. 
This result can be thought of as dealing with the 
$H^\infty$ norm of the boundary function.  
There is a similar result for $1 \leq p < \infty$, 
also due to Hardy and Littlewood, that states that for an analytic 
function $f$, the integral means
$M_p(r,f') \leq C(1-r)^{-1+\beta}$ if and only if $f \in H^p$ and 
$\|f(\cdot) - f(e^{it}\cdot)\|_{H^p} \leq C'|t|^{\beta}$.  
(See chapter 5 of \cite{D_Hp}). 

Zygmund \cite{Zygmund_Smooth-Functions} obtained similar results for the 
second iterated difference.  For example, he proved that for an 
analytic function $f$, 
one has that $f$ is continuous in $|z| \leq 1$ and 
$|f(e^{it}z) + f(e^{-it}z) - 2f(z)| \leq C|t|$ if and only if 
$|f''(z)| \leq C(1-r)^{-1}$ (see \cite{D_Hp}). 
Similar results hold for powers of $|t|$ greater than $0$ and at most 
$2$, and for 
integral means.  Analogous properties hold for harmonic functions in 
higher dimensions (see e.g.\ Chapter V of \cite{Stein_Sing-Int-Diff}). 

In \cite{GSS_Mean-Lipschitz}, the authors give results relating growth 
of area integral means of analytic functions to mean H\"{o}lder regularity 
of these functions.  In this article, we prove similar results, but 
instead use the second iterated difference, like in the result of 
Zygmund.  Also, we work on the standard weighted Bergman spaces instead 
of just the unweighted case.  For example, we prove that if 
$0 < \beta \leq 2$ and $-1 < \alpha < \infty$, and if 
$A^p_\alpha$ denotes the standard weighted Bergman space, then  
$\|f(e^{it}\cdot) + f(e^{-it}\cdot) - 2f(\cdot)\|_{A^p_\alpha} \leq 
C |t|^{\beta}$ if and only if the weighted area integral means of 
$f$ are $O((1-r)^{\beta-2})$.  
We let $\Lambda^*_{\beta, A^p_\alpha}$ denote the space of all such $f$. 
The advantage of working with the second 
iterated difference is that it can be used to characterize higher regularity 
than the difference $f(z) - f(e^{it}z)$.  

Related to this, we prove various results about the growth of 
weighted area integral means of analytic functions 
and how this relates to the 
growth of area integral means of integrals and derivatives of analytic 
functions.  We also relate growth of area integral means to growth of 
classical integral means.

The second main topic is the relation of smoothness conditions to 
extremal functions.  In \cite{Khavinson_McCarthy_Shapiro}, the 
authors give a result about mean smoothness of the solution of an
extremal problem for 
Bergman spaces.  Their work is based on 
\cite{Shapiro_Regularity-closest-approx}, where a similar result is 
given for Hardy spaces.  
Actually, many techniques in these papers that are relevant to this 
paper 
are very general and only use the fact that the spaces in question are closed 
subspaces of $L^p$ that are invariant under translations of the form 
$z \mapsto ze^{it}$.

We derive a result similar to the one in \cite{Khavinson_McCarthy_Shapiro} 
for another type of extremal problem in 
weighted Bergman spaces.  In particular, given a 
$k \in  A^q_\alpha$, for $1 < q < \infty$, the extremal problem in question is 
to find
$F \in A^p_\alpha$ such that $\|F\| = 1$ and 
$\Rp \int_{\mathbb{D}} f \overline{k} \, dA_\alpha$ is as large 
as possible, where $1/p + 1/q = 1$. 
Because of the uniform convexity of 
$L^p(dA_\alpha),$ such an $F$ always exists. 

Several results are known that allow one to deduce 
regularity properties of $F$ from regularity properties of $k$, and 
vice-versa.  See for example \cite{Ryabykh, tjf1, tjf2, tjf:pnoteven1}.  
Our result is of this type and says that if $0 < \beta \leq 2$ and 
$\|k(e^{it}\cdot) + k(e^{-it} \cdot) - 2k(\cdot)\|_{A^q_\alpha} 
\leq C|t|^{\beta}$ then 
$\|F(e^{it}\cdot) + F(e^{-it} \cdot) - 2F(\cdot)\|_{A^p_\alpha} 
\leq C'|t|^{\beta/p}$ for $p > 2$ and 
$\|F(e^{it}\cdot) + F(e^{-it} \cdot) - 2F(\cdot)\|_{A^p_\alpha} 
\leq C'|t|^{\beta/2}$ for $1< p < 2$.  (There is no need to consider the 
case $p=2$ because then $F$ is a multiple of $k$.)  
It is helpful to use the second iterated difference here because if we 
used only the first difference we would be restricted to 
$0 < \beta \leq 1$.  

Notice the exponent on the $|t|$ is ${\beta/2}$ for $p < 2$. 
This comes from an improvement to the techniques of 
\cite{Khavinson_McCarthy_Shapiro}, which yield
$|t|^{\beta/q}$.  We obtain this improvement by using an 
inequality from \cite{Ball_Carlen_Lieb} 
instead of Clarkson's inequalities.  The inequality 
we use gives a worse constant but a better power on $|t|$.  
The inequality is 
related to the fact that $L^p$ is $2$-uniformly convex for 
$1 < p \leq 2$, but Clarkson's inequalities only show that it is 
$q$-uniformly convex (see \cite{Ball_Carlen_Lieb}). 
This improvement is crucial for applying our results later in the 
paper.  

We then combine our results on extremal functions with our results on 
growth of area integral means.  We find two notable results.  
One is Corollary \ref{thm:weighted_continuous}, which 
applies to $A^p_\alpha$ extremal problems and 
says that if 
$k \in \Lambda^*_{2,A^q_\alpha}$ then $F$ has H\"{o}lder continuous 
boundary values if $2 \leq p < \infty$ and $-1 < \alpha < 0$ or if 
$1 < p < 2$ and $-1 < \alpha < p-2$.  

The other result is Theorem \ref{thm:pext}, which applies to extremal 
problems in unweighted Bergman spaces and says that 
if $k \in \Lambda_{2,A^p}$ and $1 < p < \infty$ then $|F|^{p-1} F' \in L^1$. 
Also $F' \in L^s$ for some $s > 1$.  (This applies to the unweighted 
case).  This is important because it allows for a more elementary 
proof of the results of 
\cite{tjf:pnoteven1}, without using results from 
\cite{Khavinson_Stessin} which rely on deep results from 
the theory of partial differential equations. (In fact, that was the 
original motivation for this paper).

We now discuss a subtlety that arises when dealing 
with area integral means of weighted 
Bergman spaces, since there seem to be two different types of definition of 
integral means possible.  Perhaps the most obvious definition for the
area integral mean of $f$ at radius $r$ is as
\[
\left((\alpha + 1)\int_{|z|<r} |f(z)|^p \, 
 (1-|z|^2)^\alpha \frac{dA(z)}{\pi}\right)^{1/p}, 
\]
or equivalently (except for an unimportant factor of $r^{2/p}$) as
\[
\left((\alpha + 1)\int_{|z|<1} |f(rz)|^p \, 
 (1-|rz|^2)^\alpha \frac{dA(z)}{\pi} \right)^{1/p},
\]
where 
$dA$ is normalized area measure.  On the other hand, we 
could define the integral mean as 
\[
\|f_r\|_{A^p_\alpha} = 
\left((\alpha + 1)\int_{|z|<1} |f(rz)|^p \, 
 (1-|z|^2)^\alpha \frac{dA(z)}{\pi} \right)^{1/p},
\]
where $f_r(z) = f(rz)$.  It seems likely that there are analytic functions 
for which these two  types of 
quantities have different orders of growth.  However, 
we prove that if one of them has growth in $O((1-r)^{\gamma})$ for 
$\gamma \leq 1$, then the so does the other.  

Throughout this paper, we often keep track of constants in inequalities.  
We do not investigate whether these constants are the best possible.  However, 
in future work, we expect to make use of the fact that explicit values for 
these constants are known (even if they are not the best possible values). 

\section{Integral Means and Area Integral Means}

This paper deals with Hardy and Bergman spaces.  See 
\cite{D_Hp} for information on Hardy spaces and 
\cite{D_Ap} and \cite{Zhu_Ap} for information on Bergman spaces.

Let $dA_\alpha (z) =  
(\alpha + 1) (1-|z|^2)^{\alpha} \, dA(z)/\pi$ be the 
standard weighted area measure, where $-1 < \alpha < \infty$.  
Let the Bergman space $A^p_{\alpha}$ be the space of all functions analytic 
in the unit disc such that 
\[
\|f\|_{A^p_\alpha} = \left( \int_{\mathbb{D}} |f(z)|^p \, dA_\alpha(z) 
                        \right)^{1/p} < \infty.
\]
We deal here mainly with the case $1 < p < \infty$. 
For $1 < p < \infty$, the dual of 
$A^p_\alpha$ is isomorphic to $A^q_\alpha$, where $q$ is the conjugate 
exponent of $p$
(see \cite{Zhu_Ap}).

If $f$ is analytic in the unit disc, we define the integral mean of 
order $p$ as
\[
M_p(r,f) = \left(\frac{1}{2\pi}
  \int_0^{2\pi} |f(re^{i\theta})|^p \, d\theta \right)^{1/p}
\]
for $0 < p < \infty$.  If $p = \infty$, we define 
$M_\infty(r,f) = \sup_{0 \leq \theta < 2\pi} |f(re^{i\theta})|$. It is known that the 
integral means increase with $r$ (see \cite{D_Hp}). 
A function is said to be in $H^p$ if its integral means of order 
$p$ are bounded, and we define $\|f\|_{H^p} = \sup_{0 \leq r < 1} M_p(r,f)$. 

We now formally define the area integral means. 
\begin{definition}
Let $f$ be in $L^p_{\rm{loc}}(\mathbb{D})$.  Define 
\[
A_{p,\alpha}(r,f) = 
\left( (\alpha + 1) \int_{\mathbb{D}} |f(rz)|^p  (1-|rz|^2)^{\alpha} 
\, \frac{dA(z)}{\pi} \right)^{1/p}
\]
and define 
\[
\widetilde{A}_{p,\alpha}(r,f) = 
   \left( (\alpha + 1)  \int_{r\mathbb{D}} |f(z)|^p (1-|z|^2)^{\alpha} 
\, \frac{dA(z)}{\pi} 
\right)^{1/p} .
\]
We will also define
\[
\widehat{A}_{p,\alpha}(r,f) = 
\left( (\alpha + 1) 
 \int_{\mathbb{D}} |f(rz)|^p (1-|z|^2)^{\alpha} \, 
\frac{dA(z)}{\pi}  \right)^{1/p}.
\]
and 
\[
\widehat{\widetilde{A}}_{p,\alpha}(r,f) = 
   \left( (\alpha + 1) 
\int_{r\mathbb{D}} |f(z)|^p (1-|z/r|^2)^\alpha  \frac{dA(z)}{\pi} 
\right)^{1/p} .
\]

\end{definition}
A change of variables shows that 
$\widetilde{A}_{p,\alpha}(r,f) = r^{2/p} A_{p,\alpha}(r,f)$, so that 
$A_{p,\alpha}(r,f) \asymp \widetilde{A}_{p,\alpha}(r,f)$ as $r \rightarrow 1^{-}$. 
Similarly, 
$\widehat{\widetilde{A}}_{p,\alpha}(r,f) = r^{2/p} \widehat{A}_{p,\alpha}(r,f).$
We also have that $\widehat{A}_{p,\alpha}(r,f) = \|f_r\|_{A^p_\alpha}$, where 
$f_r(z) = f(rz)$. 

In \cite{Zhu_Volume-Integral-Means, Zhu_Area-Integral-Means_Convex-I,
Zhu_Area-Integral-Means_Convex-II}, the authors defined area integral 
means and studied their convexity properties.  The area integral means 
they define are equal to 
$\widetilde{A}_{p,\alpha}(r,f)/\widetilde{A}_{p,\alpha}(r,1)$. 
These integral means have the same order of growth as 
$\widetilde{A}_{p,\alpha}$ for $-1 < \alpha < \infty$, so we do not consider 
them separately. 
It could be interesting to see if 
convexity results also hold for analogues of $\widehat{A}_{p,\alpha}$.

The following lemma is sometimes useful in analyzing the case 
 $\alpha > 0$. 
\begin{lemma}\label{lemma:alphagt0}
If
$0 \le \rho \le r^2$ then 
$(1-\rho^2) \le 2 (1-\rho^2/r^2)$.
\end{lemma}
This lemma is true because 
for  fixed $0 < r < 1$, the 
ratio $(1-\rho^2) / (1-\rho^2/r^2)$ is increasing. 

The following theorem gives information about the growth of integral means 
of functions when information about the growth of their area integral 
means is known. 
We exclude the  case $\beta > 0$ because no function other than  the 
zero function satisfies $A_{p,\alpha}(r,f) \le C(1-r)^{\beta}$ for 
$\beta > 0$. 
\begin{theorem}\label{thm:area-to-hardy}
Suppose $\beta \le 0$.  
Suppose first that 
$\alpha \ge 0$. 
If $A_{p,\alpha}(r,f) \le B(1-r)^{\beta}$ then
\[M_p(r,f) \le 
B (\alpha + 1)^{-1} \frac{2^{-\beta +(1+\alpha)/p} (1-r)^{\beta-(1+\alpha)/p}}
   {(1+\sqrt{r})^{(1+\alpha)/p}} 
\] 
and if 
$A_{p,\alpha}(r,f) \le B|\log(1-r)|$ then 
\[
M_p(r,f) \le 
B (\alpha + 1)^{-1} \frac{2^{(1+\alpha)/p} (1-r)^{-(1+\alpha)/p} 
  (|\log(1-r)| + \log(2))}
   {(1+\sqrt{r})^{(1+\alpha)/p}}. 
\] 
If we suppose instead that $\alpha \leq 0$, then 
if 
$A_{p,\alpha}(r,f) \le B(1-r)^{\beta}$ then
\[
M_p(r,f) \le 
(\alpha + 1)^{-1} \frac{2^{-\beta+(1+\alpha)/p} (1-r)^{\beta-(1+\alpha)/p}}
   {(1+\sqrt{r})^{1/p}(1+\sqrt{r}+r+\sqrt{r^3})^{\alpha/p}}.
\]
If 
$A_{p,\alpha}(r,f) \le B|\log(1-r)|$ then 
\[
M_p(r,f) \le 
B (\alpha + 1)^{-1} \frac{2^{(1+\alpha)/p} (1-r)^{-(1+\alpha)/p} 
  (|\log(1-r)| + \log(2))}
    {(1+\sqrt{r})^{1/p}(1+\sqrt{r}+r+\sqrt{r^3})^{\alpha/p}}. 
\] 

If $\alpha \le 0$, then the same conclusion also holds if in the hypothesis 
$A_{p,\alpha}$ is replaced with $\widehat{A}_{p,\alpha}$. Further, if 
the hypothesis of the theorem holds for all $r > R$, then the conclusion 
holds for all $r > R^2$. 
\end{theorem}

\begin{proof}
First note that if the hypothesis holds for $\alpha \le 0$ and 
with $\widehat{A}_{p,\alpha}$, it holds for $A_{p,\alpha}$ since then
$A_{p,\alpha}(r,f) \le \widehat{A}_{p,\alpha}(r,f)$. 

First assume that $A_{p,\alpha}(r,f) \le C(1-r)^{-1 + \beta}$. 
Now
\[
\begin{split}
(\alpha + 1)\int_{r^4}^{r^2} M_p^p(\sqrt{u},f) (1-u)^{\alpha} \, du &= \\ 
(\alpha + 1)\int_{r^2}^r M_p^p(t,f) (1-t^2)^{\alpha} 2t \, dt &\leq
   \widetilde{A}_{p,\alpha}^p(r,f) = 
   r^2 A_{p,\alpha}^p(r,f) 
\end{split}
\]
Suppose first that $\alpha \leq 0$ and that 
$A_{p,\alpha}(r,f) \leq B(1-r)^{\beta}$.  To simplify notation we may assume 
$B = 1$.   
Since the integral means are increasing we have 
\[
(r^2-r^4)(1-r^4)^{\alpha} M_p^p(r^2,f) \le 
(\alpha + 1)^{-1} r^2 A_{p,\alpha}^p(r,f).
\]
And thus, we have 
\[
M_p^p(r^2,f) \le 
(\alpha + 1)^{-1} (1-r)^{-(1+\alpha)} 
    \frac{A_{p,\alpha}^p(r,f)}{(1+r)(1+r+r^2+r^3)^\alpha}
\]
and so
\[
\begin{split}
M_p^p(r,f) &\leq 
(\alpha + 1)^{-1} (1-\sqrt{r})^{-(1+\alpha)} 
    \frac{(1-\sqrt{r})^{\beta p}}{(1+\sqrt{r})(1+\sqrt{r}+r+\sqrt{r^3})^\alpha} 
\\
&\leq
(\alpha + 1)^{-1} \frac{2^{-\beta p+1+\alpha} (1-r)^{\beta p-(1+\alpha)}}
   {(1+\sqrt{r})(1+\sqrt{r}+r+\sqrt{r^3})^\alpha} 
\end{split}
\]
The proofs of the other assertions are similar. 
\end{proof}

The following theorem provides a partial converse to 
Theorem \ref{thm:area-to-hardy}.
\begin{theorem}\label{thm:hardy-to-area}
If $M_p(r,f) \le B(1-r^2)^{\beta-(1+\alpha)/p}$ and $\beta < 0$ then 
\[\widetilde{A}_{p,\alpha}(r,f) 
\le B (\alpha + 1) (-\beta p)^{-1/p}(1-r^2)^{\beta}.\] 
If $\beta =0$ then $\widetilde{A}^p_\alpha \le C (\alpha + 1) |\log(1-r^2)|$.  
If instead $\beta >0$ then $f \in A^p_\alpha$. 
\end{theorem}
\begin{proof}
If $\beta < 1$ 
note that 
\[
\begin{split}
(\alpha + 1)^{-1}
\widetilde{A}_p^p(r,f) &= \int_0^r M_p^p(t,f) 2t (1-t^2)^{\alpha} \, dt \\
& \le
\int_0^r B^p(1-t^2)^{\beta p -1-\alpha}(1-t^2)^{\alpha} 2t \, dt\\ &= 
\int_0^r B^p(1-t^2)^{\beta p -1} 2t \, dt \\
&\le \frac{B^p}{-\beta p} (1-r^2)^{\beta p}.
\end{split}
\]
In the cases $\beta = 0$ or $\beta > 1$, 
a similar computation gives the result. 
\end{proof}

The following two lemmas are used for certain estimates later in the paper. 
They are from \cite{tjf:bergprojbounds}, and can be proven by using 
hypergeometric functions.  
\begin{lemma}\label{lemma:hypergeo_bound}
Suppose that $s < 1$ and $m+s > 1$ and that $k > -1$. 
Let $0 \le x < 1$. Then 
\[
\int_0^1 \frac{(1-y)^{-s}}{(1-xy)^m} y^k \, dy \le C_1(s,m,k) (1-x)^{1-s-m}
\]
where $C_1(s,m,k) < \infty$ is defined by 
\[
C_1(s,m,k) = \frac{\Gamma(k+1)\Gamma(1-s)}{\Gamma(2+k-s)} 
\max_{0 \le x \le 1} {}_2F_1( 2+k-s-m, 1-s; 2+k-s; x).
\]
If $2+k > s+m$ and $2+k > s$, then 
\[
C_1(s,m,k)
=
\frac{\Gamma(s+m-1) \Gamma(1-s)}
{\Gamma(m)}.
\]
\end{lemma}
There is another way to prove this lemma with worse constants. 
We break the integral into three pieces, one from $0$ to $1/2$ and 
one from $1/2$ to $x$ and one from 
$x$ to $1$.  
The first integral is bounded by a constant. 
Now, use the fact that if $0 \leq y \leq x$ then 
$(1/2)(1-y^2) \leq 1-xy \leq 1-y^2$ and $1-y \leq 1-y^2 \leq 2(1-y)$ to 
bound the second integral by
\[
C \int_{1/2}^x (1-y)^{-s-m} \, dy.
\]
Now use the fact that $1-x \le 1-xy$ to bound the third integral by 
\[
C (1-x)^{-m} \int_x^1 (1-y)^{-s} \, dy.
\]

\begin{lemma} \label{lemma:1minusrep}
Let $p>1$ and $0<r<1$. Then 
\[
\begin{split}
\frac{1}{2\pi} \int_0^{2\pi} \frac{1}{|1-re^{i\theta}|^p} d\theta &= 
(1-r^2)^{1-p} {}_2F_1\left(1-\frac{p}{2}, 1-\frac{p}{2};1;r^2\right)
\\ &\le 
\frac{\Gamma(p-1)}{\Gamma(p/2)^2} (1-r^2)^{1-p}.
\end{split}
\]
If $p=2$ we have 
\[
\frac{1}{2\pi} \int_0^{2\pi} \frac{1}{|1-re^{i\theta}|^2} d\theta = 
(1-r^2)^{-1}.
\]
\end{lemma}

We are now in a position to prove that $A_{p,\alpha} = O((1-r)^{\beta})$ 
if and only if 
$\widehat{A}_{p,\alpha} = O((1-r)^{\beta})$. 
\begin{lemma}
Let $\beta < 0$ and let $\alpha < 0$. 
If $M_p(r,f) \le C(1-r^2)^{\beta-(1+\alpha)/p}$ then 
$\widehat{A}_{p,\alpha}(r,f) 
\le C (\alpha + 1) 
C_1(-\alpha,1+ \alpha -\beta p,0)^{1/p} (1-r^2)^{\beta}$,
where $C_1$ is defined in Lemma \ref{lemma:hypergeo_bound}.  
\end{lemma}
\begin{proof}
Let $\gamma = -(\beta-(1+\alpha)/p)p = 1+\alpha  - \beta p$.  
Then we have 
\[
(\alpha + 1)^{-1}
\widehat{A}_p(r,f) \le 
 C \int_0^1 \frac{(1-t^2)^{\alpha}}{(1-r^2t^2)^{\gamma}} 2t\, dt = 
 C \int_0^1 \frac{(1-u)^{\alpha}}{(1-r^2u)^{\gamma}} \, du
\]
But by Lemma \ref{lemma:hypergeo_bound}, the integral 
is at most 
$C C_1(-\alpha,\gamma,0) (1-r^2)^{1+\alpha-\gamma} = 
C C_1(-\alpha,\gamma,0)(1-r^2)^{-1 + \beta}$. 
\end{proof}

\begin{theorem}
Suppose that $\beta \le 0$.  
There exists a constant $C > 0$ such that $A_p(r,f) \le C(1-r^2)^{\beta}$ 
if and only if
there exists a constant $C > 0$ such that 
$\widehat{A}_p(r,f) \le C(1-r^2)^{\beta}$.
\end{theorem}
\begin{proof}
The proof is clear in the case $\alpha = 0$ since the two integral 
means are equal.  

Consider now $\alpha < 0$. 
For the case $\beta = 0$, note that since the weight and the integral 
means are both increasing, 
both  
$A_p(r,f)$ and $\widehat{A}_p(r,f)$ are increasing and 
the monotone convergence theorem shows that both approach 
$\|f\|_{A^p_\alpha}$, which proves the theorem in this case.
 
Now consider $\beta < 0$. 
One direction is clear because $\widehat{A}_p(r,f) \ge A_p(r,f)$.  
Note that if 
$A_p(r,f) \le C(1-r^2)^{\beta}$ then 
$M_p(r,f) \le C(1-r^2)^{\beta-(1+\alpha)/p}$, and 
$\widehat{A}_p(r,f) \leq C(1-r^2)^{\beta}$ by the lemma.

Now, for $\alpha > 0$, we have that 
$\widehat{A}_{p,\alpha}(r,f) \le A_{p,\alpha}(r,f)$ because the weights are 
decreasing.  Also note that for fixed $0 < r < 1$, 
Lemma \ref{lemma:alphagt0} shows that 
if $0 \le \rho \le r^2$ then 
$(1-\rho^2) \le 2 (1-\rho^2/r^2)$. 
Thus 
\[
\begin{split}
\widetilde{A}_{p,\alpha}(r^2,f) &= 
\int_0^{r^2} M_p^p(\rho,f) (1-\rho^2)^\alpha 2\rho \, (\alpha + 1) d\rho
\\
&\le
2^{\alpha}\int_0^{r^2} M_p^p(\rho,f) (1-\rho^2/r^2)^\alpha 2\rho \, (\alpha + 1) 
   d\rho
\\
&\le
2^{\alpha}\int_0^{r} M_p^p(\rho,f) (1-\rho^2/r^2)^\alpha 2\rho \, (\alpha + 1) 
    d\rho
\\
&=
2^{\alpha} \widehat{\widetilde{A}}_{p,\alpha}(r,f).
\end{split}
\]
So if $\widehat{A}_{p,\alpha}(r,f) \le C (1-r^2)^{\beta}$ then 
\[
A_{p,\alpha}(r,f) \le 2^{\alpha} \widehat{A}_{p,\alpha}(\sqrt{r},f) \le 
2^{\alpha-\beta}(1-r)^{\beta} 
\le 2^{\alpha + 1 - 2\beta} (1-r^2)^{-1 + \beta}.
\] 
\end{proof}

\section{Bergman Integral Means and Derivatives}
We now discuss the relation between the area integral means of a function 
and the area integral means of its derivative and antiderivative. 
Many of the results in this section can be proved by 
changing to classical integral means and using the corresponding results for 
classical 
integral means and the theorems of the previous section, but we provide 
direct proofs. 

We let $f_s$ denote the dilation of $f$, defined by $f_s(z) = f(sz)$.  
Thus, $(f')_s(z) = f'(sz)$ and $(f_s)'(z) = sf'(sz) = s (f')_s(z)$.

\begin{theorem}\label{thm:area-deriv-to-f}
Let $f$ be analytic in the disc.  Then 
\[
\widehat{A}_{p,\alpha}(r,f-f(0)) \le 
   r \int_0^1 \widehat{A}_{p,\alpha}(r,(f')_s) \, ds.
\]
If also $\beta < -1$ and 
$\widehat{A}_{p,\alpha}(r,f') = O((1-r)^{\beta})$ then 
$\widehat{A}_{p,\alpha}(r,f) = O((1-r)^{1+\beta})$. 

If also $\beta > -1$ and 
$\widehat{A}_{p,\alpha}(r,f') = O((1-r)^{\beta})$ then 
$f \in A^p_\alpha$. 

All the above statements are true if $\widehat{A}$ is replaced 
with any of the other integral means.  
\end{theorem}

\begin{proof}
Let $z = te^{i\theta}$ and suppose that $t < r$. 
Assume without loss of generality that $f(0) = 0$. 
Note that 
\[
|f(z)| \le \int_0^t |f'(\rho e^{i\theta})| d \rho = 
           \int_0^1 |f'(sz)| t \, ds 
\le r \int_0^1 |f'(sz)| \, ds.
\]
Now use Minkowski's inequality to conclude that 
\[
\widehat{A}_{p,\alpha}(r,f) \le 
   r \int_0^1 \widehat{A}_{p,\alpha}(r,f'(s \cdot)) \, ds
\]
The same proof words for the other integral means. 
The results about order of growth follow immediately. 
\end{proof}

\begin{theorem}
Let $\beta \ge 0$. 
If $\widehat{\widetilde{A}}_{p,\alpha}(r,f) 
= O((1-r)^{-\beta})$  then 
\(
\widehat{\widetilde{A}}_{p,\alpha}(r,f') = %
   O((1-r)^{-1-\beta}).
\)
\end{theorem}
\begin{proof}
We may assume without loss of generality that $f(0) = 0$. 
Let $z = re^{i\theta}$, where $0 < r < 1$.  
Note that 
\[
f'(\rho z) = \frac{1}{2\pi i} \int_{|\zeta| = \rho} 
                 \frac{f(\zeta)}{(\zeta-\rho z)^2} d\zeta.
\]
Now
parametrize the integral with $\zeta = \rho e^{i(t+\theta)}$ to see that 
\[
\begin{split}
f'(\rho r e^{i\theta}) &= 
\frac{1}{2\pi} \int_0^{2\pi} 
      \frac{f(\rho e^{i(t+\theta)}) \rho e^{i(t-\theta)}}
             {(\rho e^{it} - \rho r)^2} \, dt \\
&=
\frac{1}{2\pi} \int_0^{2\pi} 
      \frac{f(\rho e^{i(t+\theta)}) \rho^{-1} e^{i(t-\theta)}}
             {(e^{it} - r)^2} \, dt
\end{split}
\]

Let $g(z) = f(z)/z$. 
Apply the integral 
$\int_0^r \int_0^{2\pi} 2\rho (1-\rho^2/r^2)^{\alpha} \, d\theta \, d\rho$ 
to both 
sides of the above equation and use Minkowski's inequality to conclude 
that
\[
\begin{split}
\widehat{\widetilde{A}}_{p,\alpha}(r,(f')_r) \le 
   \frac{1}{2\pi} \int_0^{2\pi} 
    \frac{\widehat{\widetilde{A}}_{p,\alpha}(r,g)}{|e^{it}-r|^2} \, dt 
&\le \widehat{\widetilde{A}}_{p,\alpha}(r,g) (1-r^2)^{-1} \\
&= 
r^{2/p} \widehat{A}_{p,\alpha}(r,g) (1-r^2)^{-1}.
\end{split}
\]
The last inequality follows from 
Lemma \ref{lemma:1minusrep}. 

If $h \in A^p_{\alpha}$ with norm $1$, then   
because the integral means 
increase, we have that $\|zh\|_{A^p_\alpha}$ is minimized when 
$h = 1$.  
Let $M_{\alpha, p} = \|z\|_{A^p_\alpha}$.  Then the above argument implies 
that $\widehat{A}_{p,\alpha}(r,g) = \|f(rz)/(rz)\|_{A^{p}_\alpha} \leq 
M_{\alpha,p}^{-1}r^{-1} \|f(rz)\|.$
Thus for large enough $r$ we have that
$\widehat{\widetilde{A}}_{p,\alpha}(r,(f')_r) \leq 
C \widehat{\widetilde{A}}_{p,\alpha}(r,f) (1-r^2)^{-1}$ for some constant $C$. 

A change of variables shows that 
$\widehat{\widetilde{A}}_{p,\alpha}(r,(f')_r) = 
r^{-2} \widehat{\widetilde{A}}_{p,\alpha}(r^2, f')$. 
Thus
\[
\widehat{\widetilde{A}}_{p,\alpha}(r^2,f') \le 
C \widehat{\widetilde{A}}_{p,\alpha}(r,f) (1-r^2)^{-1}.
\]

So if $\widehat{\widetilde{A}}_{p,\alpha}(r,f) \le C (1-r^2)^{-\beta}$, then 
$\widehat{\widetilde{A}}_{p,\alpha}(r^2,f') \le C (1-r^2)^{-1-\beta}$.  
\end{proof}

\section{Bergman Mean H\"{o}lder Functions}
We now come to the relation between mean smoothness of functions and 
the growth of their area integral means.  The results in this section are 
similar to some of those in 
\cite{GSS_Mean-Lipschitz}, except we use the second iterated difference 
instead of the first. 
The proof techniques are a combination of those in 
\cite{GSS_Mean-Lipschitz} and the techniques of Zygmund in 
\cite{Zygmund_Smooth-Functions} (see also Chapter 5 of \cite{D_Hp}). 

\begin{definition}
Let $f \in A^p_\alpha$.  Suppose 
\[
\|f(e^{it} \cdot) + f(e^{-it} \cdot) - 2f(\cdot) \|_{p,\alpha} 
\le C |t|^{\beta}
\]
for some constant $C$.  
We then say that $f \in \Lambda^*_{\beta, A^p_\alpha}$.  Furthermore, we 
define $\|f\|_{\Lambda^*,\beta,A^p_\alpha}$ 
to be the infimum of the constants $C$ such 
that the above inequality holds. 

We define $\Lambda^*_{\beta, H^p}$ similarly, but instead use the 
$H^p$ norm.
\end{definition}

We now prove Theorems \ref{thm:lip-to-bergman-growth} and 
\ref{thm:bergman-growth-to-lip}, which identify the classes 
$\Lambda^*_{\beta,A^p_\alpha}$ with the classes of functions whose second 
derivatives have area integral means with a certain order of growth. 
Both of the theorems are known to be true if we replace 
$\Lambda^*_{\beta,A^p_\alpha}$ with $\Lambda^*_{\beta, H^p}$ and area integral 
means with classical integral means.  
(For further information and corresponding results in higher dimensions, see 
Chapter V of \cite{Stein_Sing-Int-Diff}). 

The proof of this theorem and the one following are similar to the proof 
of Theorem 5.3 in \cite{D_Hp}, and also bear some resemblance to 
techniques from \cite{GSS_Mean-Lipschitz}. 

\begin{theorem}\label{thm:lip-to-bergman-growth}
Let $f \in A^p_\alpha$. Suppose that 
$\|f\|_{\Lambda^*,\beta, A^p_\alpha} < B$ for some $0 < \beta \le 2$.  Then 
$A_{p,\alpha}(r,f'') = \leq C B(1-r)^{\beta-2}$ where $C$ depends only 
on $\alpha$, but not on $\beta$. 

If in addition we have $0 < \beta < 1$ then 
$A_{p,\alpha}(r,f') = O((1-r)^{\beta-1})$ and 
if also $-1 < \alpha \leq 0$ then
$\widehat{A}_{p,\alpha}(r,f') \leq 382.5B(1-\beta)^{-1} (1-r)^{-1+\beta}$ for 
$R < r < 1$, where $R$ is some universal constant. 

If $1 < \beta < 2$ we have $f' \in A^p_{\alpha}$. 

If $\beta = 1$ we have $A_{p,\alpha}(r,f') = O(|\log(1-r)|)$ and 
if also $-1 < \alpha \leq 0$ we have 
$\widehat{A}_{p,\alpha}(r,f') \leq 382.5B |\log(1-r)|$  for 
$R < r < 1$, where $R$ is some universal constant.

\end{theorem}

\begin{proof}
By the Poisson integral formula,
\[
f(\rho z) = \frac{1}{2\pi} \int_0^{2\pi} 
   P(\rho, \theta -t) f(re^{it}) \, dt
\]
where $P$ is the Poisson kernel and $0 < \rho < 1$, and 
$z = r e^{i\theta}$.  Thus 
\[
\begin{split}
f_{\theta \theta}(\rho z) &= 
 \frac{1}{2\pi} \int_0^{2\pi} P_{22}(\rho, \theta-t) f(re^{it}) \, dt \\
&=
  \frac{1}{2\pi} \int_0^{2\pi} P_{22}(\rho, -t) f(re^{i(\theta+t)}) \, dt \\
&=
 \frac{1}{2\pi} \int_0^{\pi} P_{22}(\rho, t) \left(
f(re^{i(\theta+t)}) + f(re^{i(\theta-t)}) - 2f(re^{i\theta}) \right) \, dt.
\end{split}
\]
In the last step, we have used the fact that $P_{22}$ is even, and that 
$f(re^{i\theta})$ is constant with respect to $t$ 
and thus has second $\theta$ derivative of 
$0$.

Let $\delta = 1-\rho$.
Since $P(r,t) = (1-r^2)/(1-2r\cos(t)+r^2)$, we have 
\[
P_{tt} = \frac{8r^2 (1-r^2) \sin^2(t)}{(1-2r \cos(t) + r^2)^3}
         - \frac{2r (1-r^2)\cos(t)}{(1-2r \cos(t) + r^2)^2}.
\]
Now for $0 \le t \le \pi$ we have 
$1-2r\cos(t) + r^2 = (1-r)^2 + 4r\sin^2(t/2) \ge (1-r)^2 + 4rt^2/\pi^2$,
and thus
\[
|P_{tt}| \le 
  \frac{16(1-r)t^2}{[(1-r)^2 + 4\pi^{-2} rt^2]^3} + 
  \frac{4(1-r)}{[(1-r)^2 + 4 \pi^{-2} rt^2]^2}.
\]
Thus we have that 
\[
\begin{split}
|P_{tt}| &\le 4^{-1}r^{-3}\pi^6 (1-r)t^{-4} + 4^{-1}r^{-2}\pi^4(1-r)t^{-4} \\
&\le 4^{-1}(\pi^6 + \pi^4 r)(1-r)r^{-3}t^{-4}.
\end{split}
\] 
Let $C_1(r) = (3\pi^6 + 3 \pi^4 r)/(24\pi r^3)$.

We also wish to find the maximum for all $t$.  Note that the second term of 
in the estimate of $|P_{tt}|$ is maximized if $t=0$, and the maximum is 
$4(1-r)^{-3}$.  The first term is more difficult to maximize.  However, 
calculus shows that for fixed $y > 0$ 
the maximum of $x/(y+x)^3$ for $x \ge 0$ occurs at $x=y/2$.  Taking 
$y=(1-r)^2$ and $x=4\pi^{-2}rt^2$ shows that the maximum in the 
first term occurs for 
$4\pi^{-2}rt^2 = (1-r)^2/2$, and the maximum is 
\[
\frac{2}{3}\frac{16(1-r)(2^{-3}\pi^2(1-r)^2r^{-1})}{(1-r)^6} 
\le \frac{4\pi^2}{3} r^{-1} (1-r)^{-3}.
\] 
So \[
|P_{tt}| \le \frac{4 \pi^2 r^2+12r^3}{3r^3} (1-r)^{-1}.
\]
Let $C_2(r) = (2\pi^2r^2 + 6r^3)/(3\pi r^3)$. 

Now Minkowski's inequality gives
\[
\begin{split}
\|f_{\theta \theta}(\rho \cdot)\|_{A^p_\alpha} &\leq 
\frac{B}{2\pi} \int_0^\pi |P_{22}(\rho,t)| |t|^{\beta} \, dt \\
 &\leq 
BC_2(r) \int_0^{\delta} \delta^{-3} t^\beta\, dt + 
  BC_1(r) r^3 \int_{\delta}^{\pi} \delta t^{-4} t^{\beta} \, dt \\
&\le BC_3(r) \delta^{-2 + \beta}
\end{split}
\]
where $C_1(\rho)/(3-\beta) + C_2(\rho)/(\beta + 1) \le C_3(\rho) = 
C_1(\rho) + C_2(\rho)/2 =
(3\pi^6 + 3\pi^4\rho + 8\pi^2 \rho^2 + 24\rho^3)/(24\pi \rho^3)$. 
Thus 
\[
\|f_{\theta \theta}(r \cdot)\|_{A^p_\alpha} \le BC_3(r) (1-r)^{-2 + \beta}.
\] 
This bound goes to $\infty$ as $r \rightarrow 0$, but we can then use the 
fact that $\|f_{\theta \theta}(r \cdot)\|_{A^p_\alpha}$ is increasing to 
see that $\|f_{\theta \theta}(r \cdot)\|_{A^p_\alpha}$ is bounded by  
\[
BC_3(\max(r,1/2))(1-r)^{-2 + \beta},
\]
where the constant $C$ is independent of $r$. 

Now note that  because $f$ is analytic we have
$f_\theta = i z f'(z)$, which is itself an analytic function, so 
$f_{\theta \theta} = i z (i z f'(z))' = -z^2 f''(z) - z f'(z)$.  
Thus 
\begin{equation}\label{fprimeprime_from_ftt}
f''(z) = - z^{-2}(f_{\theta \theta} +  zf'(z))
\end{equation}
and 
\begin{equation}\label{fprime_from_ftt}
f'(z) = (1/(iz))\int_0^z f_{\theta \theta}(\zeta)/(i\zeta) \, d\zeta.
\end{equation} 
Let $g \in A^p_{\alpha}$ with norm $1$.  
Because the integral means %
increase, we have that $\|zg\|_{A^p_\alpha}$ is minimized when 
$g = 1$.  
Call this value $M_{\alpha,p}$. 

Since
$\|f_{\theta \theta}(r \cdot)\|_{A^p_\alpha} = O((1-r)^{-2 + \beta})$, 
equation \eqref{fprime_from_ftt}, the 
finiteness of $M_{\alpha, p}$, and 
Theorem \ref{thm:area-deriv-to-f} shows that $|f'(re^{i\theta})|$ is 
$O((1-r)^{-1.5+\beta})$ if $0 < \beta \leq 1$ and 
$|f'(re^{i\theta})|$ is $O(1)$ if $1 < \beta \leq 2$. 
(The number $1.5$ is not essential here, as any number between 
$1$ and $2$ would work.) 
Note also that the implied constant can be chosen independently of 
$\beta$ by using either the estimate 
$\|f_{\theta \theta}(r \cdot)\|_{A^p_\alpha} = O((1-r)^{-2 + \beta})$ or the estimate 
$\|f_{\theta \theta}(r \cdot)\|_{A^p_\alpha} = O((1-r)^{-2.5 + \beta})$. 

By 
equation \eqref{fprimeprime_from_ftt}, we have that 
$\|f''(r \cdot)\|_{A^p_\alpha} \le C (1-r)^{-2 + \beta}$ for 
large enough $r$, where $C$ is some constant, which again may 
be chosen independently of $\beta$.  
The statement about the order of growth of $f'$ follows from 
Theorem \ref{thm:area-deriv-to-f}.

We now compute the constants more explicitly in the case 
$-1 < \alpha \leq 0$. 
Note that for fixed $p$, $M_{\alpha, p}$ is 
minimized when $\alpha = 0$.  For $\alpha = 0$, the quantity is minimized 
when $p=1$, and it is then $2/3$. 
So 
$\|f_{\theta \theta}(r z)/(rz)\|_{A^p_\alpha} \le M_{\alpha,p}
C_3(r) B (1-r)^{-2 + \beta}$ for large enough $r$. 
Also note that $C_3(r) \rightarrow C_3(1)$ as $r \rightarrow 1$, and that 
$C_3(1) < 42.47$.  The fact that 
$\|f_{\theta \theta}(r z)/(rz)\|_{A^p_\alpha}$ is increasing 
now shows that 
$
\|f_{\theta \theta}(r z)/(rz)\|_{A^p_\alpha}
\le M_{\alpha,p}
C_3(1) B (1-r)^{-2 + \beta} + o((1-r)^{-2 + \beta})$. 
Also, the implied constant in the 
$o((1-r)^{-2 + \beta})$ does not depend on $\alpha$, $\beta$, or $f$.

By Theorem \ref{thm:area-deriv-to-f} we have 
\[
\begin{split}
\widehat{A}_{p,\alpha}(r,i z f') &\leq o((1-r)^{-1 + \beta}) + 
\int_{0}^1 (3/2) C_3(1) B (1-r \rho)^{-2+\beta} \, d\rho \\
& \leq 
(3/2)C_3(1) B (-1 + \beta)^{-1} (1-r)^{-1 + \beta} + 
o((1-r)^{-1 + \beta})
\end{split}
\]
for $0 < \beta < 1$ and similarly 
$\widehat{A}_{p,\alpha}(r,i z f') \le 
(3/2)C_3(1) B |\log(1-r)| + o(|\log(1-r)|)$ for $\beta = 1$. 

Thus, since $M_{\alpha, p} \leq 3/2$, we have 
\[
\widehat{A}_{p,\alpha}(r,f') \le 
95.6 B (-1 + \beta)^{-1} (1-r)^{-1 + \beta} + 
o((1-r)^{-1 + \beta})
\]
for $0 < \beta < 1$ and similarly 
$\widehat{A}_{p,\alpha}(r,f') \le 
95.6 B |\log(1-r)| + o(|\log(1-r)|)$ for $\beta = 1$. 
\end{proof}

\begin{corollary}\label{thm:lip-to-hardy-growth}
Let $f \in A^p_\alpha$ for 
$\alpha \leq 0$.   Suppose 
$\|f(e^{it} \cdot) + f(e^{-it} \cdot) - 2f(\cdot) \|_p \le B |t|^{\beta}$ 
for some $0 < \beta < 1$.  Then for all sufficiently large $r$, one has
\[
\begin{split}
M_p(r,f') &\le  \frac{2^{1-\beta-\alpha/p}}{1-\beta}
95.6 C(1-r)^{-1+\beta-(1+\alpha)/p} + 
  o((1-r)^{-1+\beta - (1+\alpha)/p})\\
&\leq 383 (1-\beta)^{-1} B (1-r)^{-1 + \beta - (1 + \alpha)/p} + 
     o((1-r)^{-1+\beta - (1+\alpha)/p}).\\
\end{split}
\]

If instead $\beta = 1$ we have 
\[M_p(r,f') \leq \frac{192}{1-\beta} B \log(1-r) (1-r)^{(-1+\alpha)/p} + 
o(\log(1-r) (1-r)^{(-1+\alpha)/p}). \]
\end{corollary}
\begin{proof}
This follows from the above theorem and Theorem 1.2.
\end{proof}

The next corollary is interesting because 
it is known that the condition 
$|\phi(x+t) + \phi(x-t) - 2\phi(x)| \le C|t|^{\beta}$ is not enough to 
guarantee that a function is measurable, much less H\"{o}lder continuous 
of some order (see \cite{D_Hp}, page 72). 
\begin{corollary}
If we exclude the condition $f \in A^p_\alpha$ in the definition of  
$\Lambda^*_{\beta,A_\alpha^p}$, it makes no difference.  
\end{corollary}
\begin{proof}
Suppose $\|f(e^{it}z) + f(e^{-it}z) - 2f(z)\|_{p,\alpha} \leq B |t|^\beta$. 
Since for each $r$, the dilation $f_r \in A^p_\alpha$ and satisfies 
$\|f_r(e^{it}z) + f_r(e^{-it}z) - 2f_r(z)\|_{p,\alpha} \leq B |t|^\beta$, 
we can apply the 
above theorem to it to see that 
$\widehat{A}_{p,\alpha}(r,(f_r)'') \leq CB(1-r)^{\beta - 2}$ where $C$ is 
independent of $r$.  
But $\widehat{A}_{p,\alpha}(r,(f_r)'') = r^2 \widehat{A}_{p,\alpha}(r^2,f'')$. 
So then 
$\widehat{A}_{p,\alpha}(r,f'') \leq CB(1-\sqrt{r})^{\beta-2}$. 
Thus 
$\widehat{A}_{p,\alpha}(r,f_r) \leq CB(1-r)^{\beta-2}$, which implies that 
$f \in A^{p}_\alpha$. 
\end{proof}

The next theorem is the converse of Theorem 
\ref{thm:lip-to-bergman-growth}. The proof is very similar to that 
of Theorem 5.3 of \cite{D_Hp} (see also \cite{Zygmund_Smooth-Functions}). 

\begin{theorem}\label{thm:bergman-growth-to-lip}
Suppose $0 < \beta \leq 2$.  
If $\widehat{A}_{p,\alpha}(f'',r) \le B(1-r)^{-2+\beta}$ then 
$f \in \Lambda_{\beta,A^p_\alpha}$. If we assume that 
$f'(0) = 0$ then 
\[
\|f\|_{\Lambda_\beta, A^p_\alpha} \le \left(48\pi^2+12\pi + \frac{1}{\beta}\right) B.
\]
\end{theorem}

\begin{proof}
We may assume without loss of generality that $f(0) = 0$, since it does 
not affect any bounds in the statement of the theorem.  
(Note however that $e^{it}+e^{-it}-2 = 2(\cos(t)-1)$ so the 
value of $f'(0)$ does affect $\|f\|_{\Lambda_\beta, A^p_\alpha}.$) 
To simplify notation, we will assume without loss of generality that 
$B=1$. 
Let $(\Delta_h f)(z) = f(e^{ih}z)-f(z)$.  So we are required to show that 
$\|\Delta_{-h} \Delta_h f\|_{p,\alpha} = O(h^{\beta})$.  

First write $\Delta_{-h} \Delta_h f = 
\Delta_{-h} \Delta_h (f-f_\rho) + \Delta_{-h} \Delta_h f_\rho$
where $f_\rho(z) = f(\rho z)$, and $0 < \rho < 1$ is a positive number 
that will be chosen later.  

Now note that integration by parts shows that 
\[
f(ze^{it}) - f(\rho z e^{it}) = 
  (1-\rho) ze^{it} f'(\rho z e^{it}) + 
     \int_{\rho}^1 ze^{it}(1-r) f''(rze^{it})ze^{it} \, dr
\]
Notice that if we apply the $A^{p}_{\alpha}$ norm to the integral 
(with respect to $z$) and use Minkowski's inequality, we can 
bound the integral by 
\[
\int_{\rho}^1 (1-r) (1-r)^{-2+\beta} \, dr \le \beta^{-1}(1-\rho)^{\beta}.
\]

We let $\Delta$ apply to the variable $t$.  
Note that 
\[
\begin{split}
\Delta_h ze^{it}f'(\rho z e^{it}) &= 
\Delta_h (ze^{it}) f'(\rho ze^{it})  + 
    ze^{i(t+h)} \Delta_h f'(\rho ze^{it}) \\
&=
\Delta_h (ze^{it}) f'(\rho ze^{it})  + ze^{i(t+h)}
\int_0^h f''(\rho z e^{it} e^{i\theta}) i \rho z e^{it}e^{i\theta} \, d \theta 
\\ &= \text{I} + \text{II}.
\end{split}
\]
The term I is bounded by $|h f'(\rho z e^{it})|$. 
Note that $\|f''(\rho z)\|_{p,\alpha} \leq (1-\rho)^{-3+\beta}$ so 
$\| f'(\rho z e^{it}) \|_{p,\alpha} \le (2-\beta)^{-1} (1-\rho)^{-2+\beta}$
by
Theorem \ref{thm:area-deriv-to-f}.  
But
we also have that 
$\|f''(\rho z)\|_{p,\alpha} \leq (1-\rho)^{-2+\beta}$ so by Theorem 
\ref{thm:area-deriv-to-f}, we have  
$\| f'(\rho z e^{it}) \|_{p,\alpha} \le (\beta-1)^{-1} (1-\rho)^{-1+\beta}.$
Thus in either case 
$\| f'(\rho z e^{it}) \|_{p,\alpha} \le 2 (1-\rho)^{-2+\beta}$. 

Now if we take the $A^p_\alpha$ norm of term II 
with respect to $z$ and use 
Minkowski's inequality we see that it is bounded in absolute value by 
\[
\int_0^h (1-\rho)^{-2+\beta}  \, d \theta = (1-\rho)^{-2+\beta}h.
\]

Putting this all together shows that 
\[ 
\| \Delta_h[ f(ze^{it}) - f(\rho z e^{it})]\|_{p,\alpha} \le
2 h (1-\rho)^{-1+\beta} + h(1-\rho)^{-1+\beta} + \beta^{-1}(1-\rho)^{\beta}.
\]   

Now note that 
\[
-\Delta_{-h} \Delta_h f(z\rho e^{it}) = 
ir\rho e^{i\theta} e^{it} 
\int_0^h f'(z\rho e^{it}e^{is})e^{is} - f'(z\rho e^{it}e^{-is})e^{-is} 
              \, ds.
\]
But the above integrand equals
\begin{equation}\label{eq:lip-to-hardy-growth-integrand1}
(e^{is}-e^{-is}) f'(z\rho e^{it}e^{is}) + 
    [ f'(z\rho e^{it} e^{is}) -  f'(z\rho e^{it} e^{-is})]e^{-is}
\end{equation}
Now the $A^p_{\alpha}$ norm of the first term is bounded by 
$4s (1-\rho)^{-2 + \beta}$, as above. And the second term equals
\[
\int_{-s}^s f''(z\rho e^{it} e^{iu}) i z \rho e^{it} e^{iu} \, du
\]
But applying Minkowski's inequality shows that the 
$A_p^\alpha$ norm of the above integral is bounded by 
$\int_{-s}^s (1-\rho)^{-2+\beta} \le 2s(1-\rho)^{-2+\beta}$.  
Thus the expression in equation \eqref{eq:lip-to-hardy-growth-integrand1} 
is bounded by 
$
6s(1-\rho)^{-2 + \beta}
$.  
Therefore
\[
\|\Delta_{-h} \Delta_h f(z\rho e^{it})\|_{p,\alpha} \le 
\int_0^h 6 s(1-\rho)^{-2 + \beta} \, ds \le 3 h^2 (1-\rho)^{-2+\beta}.
\]

Putting all of this together shows that 
\[
\|\Delta_{-h}\Delta_h f( \rho e^{it} \cdot) \|_{p,\alpha} 
\le 
3(1-\rho)^{-2+\beta} h^2 + (2+1)(1-\rho)^{-1+\beta} h + 
\beta^{-1}(1-\rho)^\beta.
\]
Now, we need to choose $\rho$ in terms of $h$ so that 
$1-\rho = O(h)$ but $1-\rho \geq 0$ for $0 \leq h < 2\pi$. 
We may take $\rho = 1-h/(4\pi)$, which gives the result. 
\end{proof}

We now have the following corollary, which relates 
functions that are mean H\"{o}lder continuous with respect to the 
Bergman space norm with 
functions that are mean H\"{o}lder continuous with respect to the 
Hardy space norm.  

\begin{corollary} Let $f$ be analytic in $\mathbb{D}$, and let 
$(1+\alpha)/p < \beta < 2$ and $1 < p < \infty$.  Then 
$f \in \Lambda^*_{\beta, A^p_\alpha}$ if and only if 
$f \in \Lambda^*_{\beta-(1+\alpha)/p, H^p}$.  The 
``only if'' part of the statement also holds if $\beta = 2$.  

\end{corollary}
\begin{proof}
This follows from Theorems \ref{thm:area-to-hardy}, \ref{thm:hardy-to-area}, 
\ref{thm:lip-to-bergman-growth} and \ref{thm:bergman-growth-to-lip}, 
and the fact that the latter two theorems hold when all area integral 
means and Bergman spaces are replaced by classical integral means and 
Hardy spaces.
\end{proof}

\section{Extremal Problems}
Let 
$k \in  A^q_\alpha$ be given, where $1 < q < \infty$. 
Let 
$F \in A^p_\alpha$ be such that $\|F\| = 1$ and 
$\Rp \int_{\mathbb{D}} F \overline{k} \, dA_\alpha$ is as large 
as possible.  There is always a unique function $F$ with this property 
because $L^p(dA_\alpha)$ is uniformly convex, see for example 
\cite{tjf1}. 
The next theorem allows us to obtain knowledge about regularity of 
$F$ from knowledge about the regularity of $k$.  For similar results that 
give regularity results about $k$ from knowledge of the regularity of 
$F$, see \cite{tjf:bergprojbounds}. 

Note the assumption
$\int_{\mathbb{D}} F \conj{k} \, dA_{\alpha} =1 1$ in the statement of the 
next theorem. 
Choosing any scalar multiple of $k$ gives the same function extremal 
function $F$.  However, this assumption simplifies the notation in the 
proof.  Also, it is clear that for bounding 
$\|F\|_{\Lambda^*, \beta, A^p_\alpha}$ in terms of $B$, we must have some lower 
bound on the size of $k$. 
\begin{theorem}\label{thm:ext-regularity}
Suppose that $k \in \Lambda_{\beta,A^p_\alpha}$, and let $F$ be the extremal 
function for $k$.
Then if $2 \le p < \infty$ we have 
$F \in \Lambda_{\beta/p, A^p_\alpha}$ 
while if $1 < p \le 2$ we have
$F \in \Lambda_{\beta/2, A^p_\alpha}$.  

Furthermore, suppose that 
$\int_{\mathbb{D}} f \conj{k} \, dA_{\alpha} = 1$ and 
$\|k(e^{it}\cdot) + k(e^{-it} \cdot) - 2k(\cdot)\|_{q,\alpha} 
\leq B|t|^{\beta}$. 
If $p \ge 2$ then 
$\|F\|_{\Lambda^*, \beta, A^p_\alpha} \le 
2e^{1/e} (B/2)^{1/p}$ whereas if 
$1 < p < 2$ then 
$\|F\|_{\Lambda^*, \beta, A^p_\alpha} \le 
2(p-1)^{-1/2}(B/2)^{1/2}$.
\end{theorem}

\begin{proof}
Suppose that 
$\|k(e^{it}\cdot) + k(e^{-it} \cdot) - 2k(\cdot)\|_{q,\alpha} 
\le B |t|^{\beta}.$ 
Then if we define $\phi_t$ to be the functional associated with 
$k(e^{it}\cdot)$, and let $\phi = \phi_0$, we have 
$\|\phi_t + \phi_{-t} - 2\phi\|_{(A^p)^*} \le B|t|^{\beta}$. 
Now let $\widetilde{\phi} = (\phi_t + \phi_{-t})/2$.  Also let 
$\widetilde{F} = (F(e^{it}\cdot) + F(e^{-it}\cdot))/2$, where 
$F$ is the extremal function for $\phi$.

Thus $\| \widetilde{\phi} - \phi\|_{(A^p)^*} \le C |t|^\beta$ 
and $\|\widetilde{F}\|_p \le 1$.  Note that 
\[\int_{\mathbb{D}} F(z) \conj{k(e^{it}z)} \, dA_\alpha(z) = 
\int_{\mathbb{D}} F(e^{-it}z) \conj{k(z)} \, dA_\alpha(z)\]
 so 
$\phi(\widetilde{F}) = \widetilde{\phi}(F).$ 
But $|\widetilde{\phi}(F)| \ge 1 - \|\phi - \widetilde{\phi}\| 
\ge 1 - B|t|^\beta$. Thus $|\phi(F+\widetilde{F})| \ge 2 - B|t|^{\beta}$ 
so $\|F + \widetilde{F}\| \ge 2 - B|t|^\beta$. 

Now let $p \ge 2$. Clarkson's inequality 
states that 
\[
\|(F + \widetilde{F})/2\|^p + \|(F - \widetilde{F})/2\|^p 
  \le (\|F\|^q + \|\widetilde{F}\|^q)^{p/q}/2^{p/q}. 
\]
Let $B' = B/2$. 
This shows that, if $|t| > B'^{-1/\beta}$, then 
\[
(1-B'|t|^\beta)^p + \|(F-\widetilde{F})/2\|^p 
 \le (\|F\|^q + \|\widetilde{F}\|^q)^{p/q}/2^{p/q} \le 1.
\]

Thus $\|(F-\widetilde{F})/2\|^p \le 1 - (1-B'|t|^\beta)^p.$  
But since $(1-x)^p$ is convex one has $(1-x)^p \ge 1-px$ so 
\[
\|(F-\widetilde{F})/2\|^p \le 1 - (1-B'p|t|^\beta) = B'p|t|^\beta.
\]
Thus $\|F - \widetilde{F}\| \le 2p^{1/p} B'^{1/p} |t|^{\beta/p} 
\le 2e^{1/e} B'^{1/p} |t|^{\beta/p}$ for $|t| > B'^{-1/\beta}$. 
And one always has $\|F-\widetilde{F}\| \leq 2$, so 
$\|F-\widetilde{F}\| \leq 2B'^{1/p} |t|^{\beta/p}$ for 
$|t|>B'^{-1/\beta}$.  But 
$e^{1/e} > 1$ so we always have 
$\| F - \widetilde{F} \| \leq 2e^{1/e} B'^{1/p} |t|^{\beta/p}$. 

The proof for $1<p<2$ is similar, but we use the inequality 
\[
\|(f+g)/2\|^2 + (p-1)\|(f-g)/2\|^2 \le (\|f\|^2 + \|g\|^2)/2
\] 
from \cite{Ball_Carlen_Lieb}. (Note that in the reference the authors 
give the inequality in Proposition 3 as  
$(\|x+y\|^2 + \|x-y\|^2) /2 \ge \|x\|^2 + (p-1)\|y\|^2$, 
which gives the one we use by setting $x=f+g$ and $y=f-g$ and dividing 
by $4$.  
One could also use their inequality from Theorem 1, namely 
$(\|x+y\|^p + \|x-y\|^p) / 2^{2/p} \ge \|x\|^2 + (p-1)\|y\|^2$ and 
set $x = (F+\widetilde{F})/2$ and $y = (F-\widetilde{F})/2$, which also gives 
$\|(F + \widetilde{F})/2\|^p + (p-1)\|(F - \widetilde{F})/2\|^p 
  \leq 1. 
$, which is the same result as we get below.
) 

Letting $f = F$ and $g = \widetilde{F}$ in the displayed inequality 
above gives 
\[
\|(F + \widetilde{F})/2\|^2 + (p-1)\|(F - \widetilde{F})/2\|^2 
  \le (\|F\|^2 + \|\widetilde{F}\|^2)/2 \leq 1. 
\]
As above this yields
\[
(1-C|t|^\beta)^2 + (p-1)\|(F-\widetilde{F})/2\|^2 
 \le 1
\]
for $|t| \leq C^{-1/\beta}$. 
Thus $(p-1)\|(F-\widetilde{F})/2\|^2 \le 1 - (1-C|t|^\beta)^2.$  
As above, this shows that 
\[
\|(F-\widetilde{F})/2\|^2 = \frac{2C}{p-1}|t|^\beta.
\]
Thus $\|F - \widetilde{F}\| \le \sqrt{2} (p-1)^{-1/2} C^{1/2} |t|^{\beta/2}$ 
for $|t| \leq C^{-1/\beta}$. But since we always have 
$\|F - \widetilde{F} \| \leq 2$, we have  
$\|F-\widetilde{F}\| \leq 2C^{1/2} |t|^{\beta/2}$ for 
$|t|>C^{-1/\beta}$.  So in any event, 
$\|F-\widetilde{F}\| \leq 2(p-1)^{-1/2} C^{1/2} |t|^{\beta/2}$. 
\end{proof}

\begin{theorem}\label{thm:pext}
Let $\alpha = 0$. 
If $k \in \Lambda_{2,A^p}$ and $1 < p < \infty$ then $|F|^{p-1} F' \in L^1$. 
Also $F' \in L^s$ for some $s > 1$.  
\end{theorem}

\begin{proof}
First let $2 < p < \infty$.  If $k \in \Lambda_{2,A^p}$ then 
$F \in \Lambda_{2/p, A^p}$ by Theorem 4.1.
But this shows that $A_p(r,F') \leq C(1-r)^{2/p-1}$ by Theorem 3.1.
Thus 
$M_p(r,F') \le C/(1-r)^{2/p-1-1/p} = C/(1-r)^{1/p-1}$.
Then for small $\delta > 0$ we can use the fact that integral means 
increase with $p$ to see that 
$\|F'\|_{A^{1+\delta}}^{1+\delta} \leq \int_0^1 2r(1-r)^{(1/p-1)(1+\delta)} \, dr < 
\infty$. 

Also, $F$ is in $H^p$ by \cite{D_Hp}, Theorem 5.4 (one may also apply 
Ryabykh's theorem to see this). 
Then 
$M_1(r, |F|^{p-1} F') \le M_q(r, |F|^{p-1}) M_p(r,F') \le 
\|F\|_{H^p}^{p-1} C / (1-r)^{1-1/p}$.  
But $\| |F^{p-1}| F'\|_{L^1} = \int_0^1 M_1(r, |F|^{p-1}F') 2r\, dr \le C$. 

Now let $1 < p \le 2$. The function 
$F \in \Lambda^*_{1, A^p}$ by Theorem 4.1.
But this shows that $A_p(r,F') \leq C(1-r)^{-\epsilon}$ for any 
$\epsilon > 0$, by Theorem 3.1.
Thus $M_p(r,F') \le C/(1-r)^{-\epsilon-1/p}.$  
And we may choose 
$\epsilon$ so that $\epsilon + 1/p < 1$. 
The same reasoning as above shows that $F' \in A^{1+\delta}$ for small 
enough delta. 

Also, $F$ is in $H^p$ as above. Then 
\[
M_1(r, |F|^{p-1} F') \le M_q(r, |F|^{p-1}) M_p(r,F') \le 
\|F\|_{H^p}^{p-1} C / (1-r)^{-\epsilon-1/p}.
\]  
But 
\[
\| |F^{p-1}| F'\|_{L^1} = \int_0^1 M_1(r, |F|^{p-1}F') \, 2r dr \le C.
\] 
 
\end{proof}
In fact the same method combined with the theorem of Hardy and 
Littlewood on the comparative growth of integral means 
shows that $|F|^{p-1}F'$ is in $L^s$ for some 
$s > 1$. 

This allows us to give an alternate proof of the results of 
\cite{tjf:pnoteven1}, by providing an 
alternative proof of Lemma 1.1 which avoids using the 
regularity results of Khavinson and Stessin from \cite{Khavinson_Stessin}.  
To give more detail: 
Using the above result, 
we can prove Theorem 2.1 of \cite{tjf:pnoteven1} in exactly the same 
way as it is proved in the reference, since the proof only uses 
the fact that $|F|^{p-1}(\sgn\overline{F}) F' \in L^1$. 
Also, Theorem 3.1 in the reference follows immediately from Theorem 2.1. 
If $k$ is a polynomial, then 
taking $m > \deg(k)$ in Theorem 3.1 and using the fact $|F|^p$ is real 
valued shows that $|F|^p$ is a trigonometric polynomial, so 
$F \in H^\infty$. This and the fact that $F' \in A^s$ for some $s > 0$ 
provides an alternate proof of Lemma 1.1, which is the only place where 
the results of Khavinson and Stessin are directly cited.

The next corollary is similar to the result of Khavinson and 
Stessin from \cite{Khavinson_Stessin} about the H\"{o}lder continuity 
of extremal functions in the unweighted Bergman space, given 
enough regularity on $k$.  Our corollary applies to certain weighted 
Bergman spaces, however, and its method of proof is completely different.
It uses two well known lemmas, which we state after the proof. 

\begin{corollary}\label{thm:weighted_continuous}
Let $1 < p < \infty$ and let $p$ and $q$ be conjugate exponents. 
Suppose $k \in \Lambda^*_{2,A^q_\alpha}$.  Then 
$M_p(r,f') \le C (1-r)^{-1+2/\nu -(1+\alpha)/p}$ where 
$\nu$ is any number greater than $2$ for $1 < p < 2$ 
and $\nu = p$ for $2 \le p < \infty$.  
If $2 \le p < \infty$ and  
$-1 < \alpha < 0$, then $f$ has H\"{o}lder continuous boundary values.
If $1 < p < 2$ and $-1 < \alpha < p-2$, the same conclusion holds. 
\end{corollary}
\begin{proof}
Suppose $B > \|k\|_{\Lambda^*, 2, A^q_\alpha}$. 

First let $p > 2$. 
First apply Theorem \ref{thm:ext-regularity} to see that $F \in \Lambda_{2/p,A^{p,\alpha}}$.  
Then apply Theorem \ref{thm:lip-to-bergman-growth} to see that $A_{p,\alpha}(r,f') \le C(1-r)^{2/p-1}$. 
Then apply Theorem \ref{thm:area-to-hardy} to see that 
$M_p(r,f') \leq C(1-r)^{2/p-1-1/p-\alpha/p}$.  Then apply the 
Lemma \ref{lemma:comparative-means} see that 
$M_{\infty}(r,f') \leq C(1-r)^{2/p-1-1/p-\alpha/p-1/p} = C(1-r)^{-1-\alpha/p}$.  
If $-1 < \alpha < 0$ then 
$-1-\alpha/p > -1$, so we have that $f$ is H\"{o}lder continuous 
in the disc by Lemma \ref{lemma:means-holder}.  
The H\"{o}lder exponent is $-\alpha/p$.  The H\"{o}lder constant 
is bounded above by 
\[
2e^{1/e}(B/2)^{1/p} \cdot 383\left(1-\frac{2}{p}\right)^{-1} 
  \cdot 2\left(\frac{\Gamma(q-1)}{\Gamma(q/2)^2}\right)^{1/q} 
\cdot \left(1 - \frac{2p}{\alpha}\right).
\]

For $p < 2$, 
apply Theorem \ref{thm:ext-regularity} to see that $F \in \Lambda_{1,A^{p,\alpha}}$.  
Then apply Theorem \ref{thm:lip-to-bergman-growth} to see that $A_{p,\alpha}(r,f') \le C|\log(1-r)| 
\leq C(1-r)^{-\epsilon}$ for any $\epsilon > 0$. 
Then apply Theorem \ref{thm:area-to-hardy} to see that 
$M_p(r,f') \leq C(1-r)^{-\epsilon-1/p-\alpha/p}$.  Then apply 
Lemma \ref{lemma:comparative-means} to see that 
$M_{\infty}(r,f') \leq C(1-r)^{-\epsilon-1/p-\alpha/p-1/p} = 
C(1-r)^{-\epsilon -2/p -\alpha/p}$.  
But $-2/p-\alpha/p > -1$ if $\alpha < p-2$, 
so we have that $f$ is H\"{o}lder continuous 
in the disc by Lemma \ref{lemma:means-holder}
The H\"{o}lder exponent is $1-2/p-\alpha/p-\epsilon$.  
The constant is bounded above by 
\[
\begin{split}
2(p-1)^{-1/2}(B/2)^{1/2} &\cdot 192\left(1-\frac{2}{p}\right)^{-1} 
  \cdot 2\left(\frac{\Gamma(q-1)}{\Gamma(q/2)^2}\right)^{1/q} 
\\
&\cdot \left(1 - \frac{2}{1-2/p-\alpha/p-\epsilon}\right).
\end{split}
\]
\end{proof}

The lemmas that follow are used in the proof of the above theorem. 
Both of them are due to Hardy and Littlewood. 
\begin{lemma}[see \cite{D_Hp}, Theorem 5.1]\label{lemma:means-holder}
If $|f'(re^{i\theta})| \le C(1-r)^{1-\beta}$ for all sufficiently large 
$r$ then $f$ is continuous in the closed unit disc and 
\[
|f(e^{i\phi}) - f(e^{i\theta})| \leq 
  \left(1+\frac{2}{\beta}\right) C |\phi - \theta|^{\beta}.
\]
\end{lemma}

The constant in the next lemma follows from applying 
H\"{o}lder's inequality to the Cauchy Integral formula, and using 
Lemma 
\ref{lemma:1minusrep} to see that 
\[
\begin{split}
1/(2\pi) \int_0^{2\pi} |\rho e^{it}-r|^{-p} \, dt 
& \leq
\rho^{-p} \Gamma(p-1)/\Gamma(p/2)^2 (1-(r/\rho))^{1-p} \\
&\leq 
2^{1-p} (1+\epsilon)
(\Gamma(p-1)/\Gamma(p/2)^2)(1-r)^{1-p}.
\end{split}
\]
 for large enough $r$. 
(See the proof in \cite{D_Hp}).

\begin{lemma}[see \cite{D_Hp}, Theorem 5.9]\label{lemma:comparative-means}
Let $1 < p < \infty$. 
If for sufficiently large $r$ we have
$M_p(r,f) \le K(1-r)^{-a}$ then given any $\epsilon > 0$ there is 
an $R$ such that for $R < r < 1$ we have 
$M_\infty(r,f) \le CK (1-r)^{-(a+1/p)}$ where 
$C = \left(2(1+\epsilon)\Gamma(p'-1)/\Gamma(p'/2)^2\right)^{1/p'}$. 
\end{lemma}

%
%
%
%
%


\begin{thebibliography}{10}

\bibitem{Ball_Carlen_Lieb}
Keith Ball, Eric~A. Carlen, and Elliott~H. Lieb, \emph{Sharp uniform convexity
  and smoothness inequalities for trace norms}, Invent. Math. \textbf{115}
  (1994), no.~3, 463--482. \MR{1262940}

\bibitem{D_Hp}
Peter Duren, \emph{Theory of {$H\sp{p}$} spaces}, Pure and Applied Mathematics,
  Vol. 38, Academic Press, New York, 1970. \MR{MR0268655 (42 \#3552)}

\bibitem{D_Ap}
Peter Duren and Alexander Schuster, \emph{Bergman spaces}, Mathematical Surveys
  and Monographs, vol. 100, American Mathematical Society, Providence, RI,
  2004. \MR{MR2033762 (2005c:30053)}

\bibitem{tjf:bergprojbounds}
Timothy Ferguson, \emph{Bounds on integral means of {B}ergman projections and
  their derivatives}, arXiv:1503.04121.

\bibitem{tjf:pnoteven1}
\bysame, \emph{Extremal problems in bergman spaces and an extension of
  {R}yabykh's $h^p$ regularity theorem for $1<p<\infty$}, Indiana Univ. Math.
  J. \textbf{To appear.}

\bibitem{tjf1}
Timothy Ferguson, \emph{Continuity of extremal elements in uniformly convex
  spaces}, Proc. Amer. Math. Soc. \textbf{137} (2009), no.~8, 2645--2653.

\bibitem{tjf2}
Timothy Ferguson, \emph{Extremal problems in {B}ergman spaces and an extension
  of {R}yabykh's theorem}, Illinois J. Math. \textbf{55} (2011), no.~2,
  555--573 (2012). \MR{3020696}

\bibitem{GSS_Mean-Lipschitz}
P.~Galanopoulos, A.~G. Siskakis, and G.~Stylogiannis, \emph{Mean {L}ipschitz
  conditions on {B}ergman space}, J. Math. Anal. Appl. \textbf{424} (2015),
  no.~1, 221--236. \MR{3286557}

\bibitem{Zhu_Ap}
H{\aa}kan Hedenmalm, Boris Korenblum, and Kehe Zhu, \emph{Theory of {B}ergman
  spaces}, Graduate Texts in Mathematics, vol. 199, Springer-Verlag, New York,
  2000. \MR{1758653 (2001c:46043)}

\bibitem{Khavinson_McCarthy_Shapiro}
Dmitry Khavinson, John~E. McCarthy, and Harold~S. Shapiro, \emph{Best
  approximation in the mean by analytic and harmonic functions}, Indiana Univ.
  Math. J. \textbf{49} (2000), no.~4, 1481--1513. \MR{MR1836538 (2002b:41023)}

\bibitem{Khavinson_Stessin}
Dmitry Khavinson and Michael Stessin, \emph{Certain linear extremal problems in
  {B}ergman spaces of analytic functions}, Indiana Univ. Math. J. \textbf{46}
  (1997), no.~3, 933--974. \MR{MR1488342 (99k:30080)}

\bibitem{Ryabykh}
V.~G. Ryabykh, \emph{Extremal problems for summable analytic functions},
  Sibirsk. Mat. Zh. \textbf{27} (1986), no.~3, 212--217, 226 ((in Russian)).
  \MR{MR853902 (87j:30058)}

\bibitem{Shapiro_Regularity-closest-approx}
Harold~S. Shapiro, \emph{Regularity properties of the element of closest
  approximation}, Trans. Amer. Math. Soc. \textbf{181} (1973), 127--142.
  \MR{0320606}

\bibitem{Stein_Sing-Int-Diff}
Elias~M. Stein, \emph{Singular integrals and differentiability properties of
  functions}, Princeton Mathematical Series, No. 30, Princeton University
  Press, Princeton, N.J., 1970. \MR{0290095}

\bibitem{Zhu_Area-Integral-Means_Convex-II}
Chunjie Wang, Jie Xiao, and Kehe Zhu, \emph{Logarithmic convexity of area
  integral means for analytic functions {II}}, J. Aust. Math. Soc. \textbf{98}
  (2015), no.~1, 117--128. \MR{3294311}

\bibitem{Zhu_Area-Integral-Means_Convex-I}
Chunjie Wang and Kehe Zhu, \emph{Logarithmic convexity of area integral means
  for analytic functions}, Math. Scand. \textbf{114} (2014), no.~1, 149--160.
  \MR{3178110}

\bibitem{Zhu_Volume-Integral-Means}
Jie Xiao and Kehe Zhu, \emph{Volume integral means of holomorphic functions},
  Proc. Amer. Math. Soc. \textbf{139} (2011), no.~4, 1455--1465. \MR{2748439}

\bibitem{Zygmund_Smooth-Functions}
A.~Zygmund, \emph{Smooth functions}, Duke Math. J. \textbf{12} (1945), 47--76.
  \MR{0012691}

\end{thebibliography}

\providecommand{\bysame}{\leavevmode\hbox to3em{\hrulefill}\thinspace}
\providecommand{\MR}{\relax\ifhmode\unskip\space\fi MR }
\providecommand{\MRhref}[2]{%
  \href{http://www.ams.org/mathscinet-getitem?mr=#1}{#2}
}
\providecommand{\href}[2]{#2}

\end{document}